\documentclass[11pt]{amsart}
\usepackage{verbatim,amsmath,enumerate,amssymb,color,url,multirow}
\usepackage[unicode]{hyperref}
\usepackage[normalem]{ulem} 

\newmuskip\pFqskip
\mathchardef\pFcomma=\mathcode`, 
\newcommand*\pFq[5]{%
  \begingroup
  \begingroup\lccode`~=`,
    \lowercase{\endgroup\def~}{\pFcomma\mkern\pFqskip}%
  \mathcode`,=\string"8000
  {}_{#1}F_{#2}\biggl(\genfrac..{0pt}{}{#3}{#4};#5\biggr)%
  \endgroup
}

\newtheorem{Theorem}{Theorem}[section]
\newtheorem{Lemma}[Theorem]{Lemma}

\newtheorem{Corollary}[Theorem]{Corollary}

\theoremstyle{remark}
\newtheorem{Example}[Theorem]{\bf Example}

\renewcommand{\d}{{\mathrm d}}
\renewcommand{\Im}{\operatorname{Im}}

\begin{document}

\title{Hypergeometric modular equations}

\author{Shaun Cooper}
\address{Institute of Natural and Mathematical Sciences, Massey University\,---\,Albany, Private Bag 102904, North Shore Mail Centre, Auckland 0745, New Zealand}
\email{s.cooper@massey.ac.nz}

\author{Wadim Zudilin}
\address{Max-Planck-Institut f\"ur Mathematik, Vivatsgasse 7, 53111 Bonn, Germany}
\email{wzudilin@mpim-bonn.mpg.de}
\address{School of Mathematical and Physical Sciences, The University of Newcastle, Callaghan NSW 2308, Australia}

\dedicatory{In the memory of Jon Borwein, the late Dr Pi}

\begin{abstract}
We record $\binom{42}2+\binom{23}2+\binom{13}2=1192$
functional identities that, apart from being amazingly amusing by themselves, find applications in
derivation of Ramanujan-type formulas for $1/\pi$ and in computation of mathematical constants.
\end{abstract}

\maketitle
\numberwithin{equation}{section}

It is all about $\pi$.

\section{Introduction and statement of results}
\label{sec1}

One of the best known results of Ramanujan is his collection of series for~$1/\pi$ in~\cite{ramanujan_pi}.
One of memorable achievements of Jonathan and Peter Borwein was proving the entries in Ramanujan's collection some 70 years later \cite{agm}.
As a representative example, we quote the series \cite[Eq.~(29)]{ramanujan_pi}
\begin{equation}
\label{pi01}
\sum_{n=0}^\infty \frac{(\frac12)_n^3}{n!^3} \biggl(n+ \frac{5}{42}\biggr)
\biggl(\frac{1}{2}\biggr)^{6n}
= \frac{1}{\pi} \times \frac{8}{21}.
\end{equation}
Here and below we use
\begin{equation*}
(s)_n=\frac{\Gamma(s+n)}{\Gamma(s)}=\begin{cases}
s(s+1)\dotsb(s+n-1) & \text{if $n=1,2,\ldots$}\,, \\
1 & \text{if $n=0$},
\end{cases}
\end{equation*}
for the Pochhammer symbol (or shifted factorial), as well as the related notation
\begin{equation}
\pFq{m+1}{m}{a_0,a_1,\dots,a_m}{b_1,\dots,b_m}{x}
=\sum_{n=0}^\infty\frac{(a_0)_n(a_1)_n\dotsb(a_m)_n}{n!\,(b_1)_n\dotsb(b_m)_n}\,x^m
\label{pFq}
\end{equation}
for the generalized hypergeometric function.

Ramanujan's series for $1/\pi$ may all be expressed in the form
$$
\biggl(\alpha + x\,\frac{\d}{\d x}\biggr)
\pFq{3}{2}{\frac12,s,1-s}{1,1}{x}\bigg|_{x=x_0}=\frac{\beta}{\pi}
$$
where $\alpha$, $\beta$ and $x_0$ are algebraic numbers and $s \in \{\frac16,\frac14,\frac13,\frac12\}$.
In the example~\eqref{pi01}, we have
$$
\alpha=\frac{5}{42}, \quad \beta=\frac{8}{21}, \quad x=\frac{1}{64} \quad\text{and}\quad s=\frac12.
$$
Behind such identities there is a beautiful machinery of modular functions, something that was hinted by Ramanujan
a century ago and led to the earlier and later proofs \cite{agm,domb,Chud}.

A more recent investigation on interrelationships among Ramanujan's series in the works of several authors
\cite{aycock,mathematika,translation,rogers,Zu,Zu5} suggests considering transformation formulas of the type
\begin{equation}
\label{generaltransformation}
\pFq{3}{2}{\frac12,s,1-s}{1,1}{x(p)} = r(p)\, \pFq{3}{2}{\frac12,t,1-t}{1,1}{y(p)},
\end{equation}
where $x(p)$, $y(p)$ and $r(p)$ are algebraic functions of $p$, and $s,t \in \{\frac16,\frac14,\frac13,\frac12\}$.
For example, Aycock \cite{aycock} was mainly interested in using instances of \eqref{generaltransformation} to derive formulas such as \eqref{pi01}.
The discussion of transformation formulas of the type \eqref{generaltransformation} in \cite[pp. 15--16]{aycock}
contains just eight formulas that were obtained by searching the literature.
Just to mention a few, we list the examples\footnote{Some misprints in \cite[p. 15]{aycock} have been corrected.}
\begin{align}
\label{eg0}
\pFq{3}{2}{\frac12,\frac12,\frac12}{1,1}{x}
&= \frac{2}{\sqrt{4-x}}\,\pFq{3}{2}{\frac16,\frac12,\frac56}{1,1}{\frac{27x^2}{(4-x)^3}},
\\
\label{eg1}
\pFq{3}{2}{\frac14,\frac12,\frac34}{1,1}{\frac{256x}{(1+27x)^4}}
&= \frac{1+27x}{1+3x}\, \pFq{3}{2}{\frac14,\frac12,\frac34}{1,1}{\frac{256x^3}{(1+3x)^4}}
\\ \intertext{and}
\label{eg2}
\pFq{3}{2}{\frac13,\frac12,\frac23}{1,1}{4x(1-x)}
&= \frac{3}{\sqrt{9-8x}}\,\pFq{3}{2}{\frac16,\frac12,\frac56}{1,1}{\frac{64x^3(1-x)}{(9-8x)^3}}.
\end{align}

We should point out that algebraic transformations of hypergeometric functions, in particular, of modular origin,
are related to the monodromy of the underlying linear differential equations. This is a reasonably popular topic,
with Goursat's original 140-page contribution \cite{goursat} as the starting point.
See \cite{almkvistetal,maier,vidunas} for some recent developments.

The goal of this work is to systematically organize and classify identities of the type \eqref{generaltransformation}.
In particular, we will show that the functions that occur in \eqref{eg0}--\eqref{eg2}
are part of a single result that asserts that forty-two functions are equal.
Our results also encapsulate identities such as
\begin{equation}
\label{pb}
f_{6b}(x)=\frac{1}{1-x}f_{6c}\biggl(\frac{x}{1-x}\biggr)
\quad\text{and}\quad
f_{6c}(x)=\frac{1}{1+x}f_{6b}\biggl(\frac{x}{1+x}\biggr)
\end{equation}
where
$$
f_{6b}(x)=\sum_{n=0}^\infty \biggl\{\sum_{j+k+\ell=n} \biggl(\frac{n!}{j!k!\ell!}\biggr)^2\biggr\}x^n
$$
and
$$
f_{6c}(x)=\sum_{n=0}^\infty \biggl\{ \sum_{j+k=n}\biggl(\frac{n!}{j!k!}\biggr)^3\biggr\} x^n;
$$
that is,
$f_{6b}(x)$ is the generating function for sums of squares of trinomial coefficients, while
$f_{6c}(x)$ is the generating function for sums of cubes of binomial coefficients. Results for
Ap\'ery, Domb and Almkvist--Zudilin numbers, as well as sums of the fourth powers of
binomial coefficients, will also appear as special cases of our results.

Our results also include transformation formulas such as
\begin{multline*}
\frac{1}{(1-4x)^{5/2}}\,\pFq{3}{2}{\frac12,\frac12,\frac12}{1,1}{-64x\biggl(\frac{1+x}{1-4x}\biggr)^5}
\\
=\frac{1}{(1-4x)^{1/2}}\,\pFq{3}{2}{\frac12,\frac12,\frac12}{1,1}{-64x^5\biggl(\frac{1+x}{1-4x}\biggr)}
\end{multline*}
are part of a thirteen-function identity.

This work is organized as follows. In the next section, some sequences and their generating functions are
defined.

The main results are stated in Sections \ref{sec3}, \ref{sec4} and \ref{sec5}.
Each section consists of a single theorem that asserts that a large number of functions are equal.

Short proofs, using differential equations, are given in Section~\ref{sec6}.
Alternative proofs using modular forms, that help put the results into context, are given in Section~\ref{sec7}.

Several special cases are elucidated in Section~\ref{sec8}.
An application to Ramanujan's series for $1/\pi$, using some of Aycock's ideas, is given in Section~\ref{sec9}.

\section{Definitions and background information}
\label{sec2}

The series in \eqref{pFq} that defines the hypergeometric functions ${}_2F_1$ and ${}_3F_2$ converges for $|x|<1$.
Clausen's identity \cite[p.~116]{aar} is
$$
\biggl\{\pFq{2}{1}{a,b}{a+b+\frac12}{x}\biggr\}^2
=\pFq{3}{2}{2a,2b,a+b}{2a+2b,a+b+\frac12}{x}.
$$
It may be combined with the quadratic transformation formula \cite[p.~125]{aar}
$$
\pFq{2}{1}{2a,2b}{a+b+\frac12}{x} = \pFq{2}{1}{a,b}{a+b+\frac12}{4x(1-x)}
$$
to give
$$
\biggl\{\pFq{2}{1}{2a,2b}{a+b+\frac12}{x}\biggr\}^2
= \pFq{3}{2}{2a,2b,a+b}{2a+2b,a+b+\frac12}{4x(1-x)}.
$$
We will be interested in the special case $2a=s$, $2b=1-s$, that is
\begin{equation}
\label{clauclau}
\biggl\{\pFq{2}{1}{s,1-s}{1}{x}\biggr\}^2
=\pFq{3}{2}{\frac12,s,1-s}{1,1}{4x(1-x)}
\end{equation}
where $s$ assumes one the values $1/6$, $1/4$, $1/3$ or $1/2$.
For cosmetic reasons we introduce the following nicknames for these special instances of hypergeometric functions:
$$
f_\ell(x)=\pFq{2}{1}{s,1-s}{1}{C_sx}
$$
and
$$
F_\ell(x)=\pFq{3}{2}{\frac12,s,1-s}{1,1}{4C_sx},
$$
where $\ell=\ell_s=1,2,3,4$ and $C_s=432,64,27,16$ for $s=1/6,1/4,1/3,1/2$, respectively.
The arithmetic normalization constants $C_s$ are introduced in such
a way that the series $f_\ell(x)$ and $F_\ell(x)$ all belong to the ring $\mathbb Z[[x]]$; for example,
\begin{equation}
\begin{gathered}
F_1(x)=\sum_{n=0}^\infty\binom{6n}{3n}\binom{3n}{2n}\binom{2n}{n}x^n,
\quad
F_2(x)=\sum_{n=0}^\infty\binom{4n}{2n}{\binom{2n}{n}}^2x^n,
\\
F_3(x)=\sum_{n=0}^\infty\binom{3n}{n}{\binom{2n}{n}}^2x^n
\quad\text{and}\quad
F_4(x)=\sum_{n=0}^\infty{\binom{2n}{n}}^3x^n.
\end{gathered}
\label{bhf}
\end{equation}
The convergence domains of the series for $f_\ell(x)$ and $F_\ell(x)$ are then $|x|<1/C_s$ and $|x|\le1/(4C_s)$, respectively.

\medskip
Our further examples of the series from $\mathbb Z[[x]]$ are generating functions of so-called Ap\'ery-like sequences.
Let $\alpha$, $\beta$ and $\gamma$ be fixed, and consider the recurrence relations
\begin{equation}
\label{recw1}
(n+1)^2t(n+1)=(\alpha n^2+\alpha n+\beta)t(n)+\gamma\,n^2t(n-1)
\end{equation}
and
$$
(n+1)^3T(n+1) = -(2n+1)(\alpha n^2+\alpha n+\alpha-2\beta)T(n)-(\alpha^2+4\gamma) n^3T(n-1).
$$
We assume that $n$ is a non-negative integer in each recurrence relation, and use the single initial
condition $t(0)=T(0)=1$ to start each sequence.

\begin{table}[h]
\caption{Solutions to recurrence relations}
\begin{center}
{\scriptsize\begin{tabular}{|c|c|c|}
\hline
&&\\
 $(\alpha,\beta,\gamma)$ & $t(n)$ & $T(n)$  \\
&&\\
\hline\hline && \\
$(11,3,1)$
& $\displaystyle{\sum_{j} {\binom{n}{j}}^2\binom{n+j}{j}}$ &
${\displaystyle{\sum_{j}} (-1)^{j+n}\binom{n}{j}^3 \binom{4n-5j}{3n}}$
 \\ &&\\
$(-17,-6,-72)$ & $\displaystyle{\sum_{j,\ell}(-8)^{n-j}\binom{n}{j}\binom{j}{\ell}^3}$ &
$\displaystyle{\sum_{j} \binom{n}{j}^2\binom{n+j}{j}^2}$
\\ &&\\
$(10,3,-9)$ &
$\displaystyle{\sum_{j} \binom{n}{j}^2\binom{2j}{j}} = \sum_{j+k+\ell=n}\biggl(\frac{n!}{j!k!\ell!}\biggr)^2$ &
$\displaystyle{(-1)^{n}\sum_{j} \binom{n}{j}^2\binom{2j}{j}\binom{2n-2j}{n-j}}$
 \\ &&\\
$(7,2,8)$ &
$\displaystyle{\sum_{j} \binom{n}{j}^3}$ &
${\displaystyle{\sum_{j}} (-3)^{n-3j}\binom{n+j}{j}\binom{n}{j}\binom{n-j}{j}\binom{n-2j}{j}}$
 \\ && \\
\hline
\end{tabular}}
\end{center}
\label{table1}
\end{table}

Define the generating functions
$$
f(x) = \sum_{n=0}^\infty t(n)x^n,  \quad
F(x) = \sum_{n=0}^\infty \binom{2n}{n} t(n)x^n \quad\text{and}\quad
G(x) = \sum_{n=0}^\infty T(n)x^n.
$$
It is known \cite{almkvistetal,ctyz} that
\begin{align}
\label{zs}
f(x)^2
&=\frac{1}{1+\gamma x^2}\,F\biggl(\frac{x(1-\alpha x-\gamma x^2)}{(1+\gamma x^2)^2}\biggr)
\\
&=\frac{1}{1-\alpha x-\gamma x^2}\,G\biggl(\frac{x}{1-\alpha x-\gamma x^2}\biggr).
\nonumber
\end{align}
When $\gamma=0$, the first equality in \eqref{zs} reduces to \eqref{clauclau}.

Let $f_{6a}(x)$, $F_{6a}(x)$ and $G_{6a}(x)$ be the functions $f(x)$, $F(x)$ and $G(x)$ respectively, for the parameter
values
$$
(\alpha,\beta,\gamma)=(-17,-6,-72).
$$
Similarly, let $f_{6b}(x)$, $F_{6b}(x)$ and $G_{6b}(x)$ be the respective functions that correspond to the values
$$
(\alpha,\beta,\gamma)=(10,3,-9),
$$
while $f_{6c}(x)$, $F_{6c}(x)$ and $G_{6c}(x)$ are the respective functions that correspond to the values
$$
(\alpha,\beta,\gamma)= (7,2,8)
$$
and $f_5(x)$, $F_5(x)$ and $G_5(x)$ are the respective functions that correspond to the values
$$
(\alpha,\beta,\gamma)= (11,3,1).
$$
Formulas for the coefficients $t(n)$ and $T(n)$ that involve sums of binomial coefficients are known
in the four special cases defined above, and these are listed in Table~\ref{table1}.
The entries for $t(n)$ come from a list that is originally due to D.~Zagier \cite[Section 4]{zagier}.
It may also be mentioned that the numbers $T(n)$ in the cases $(\alpha,\beta,\gamma)=(-17,-6,-72)$, $(10,3,-9)$
and $(7,2,8)$ are called the Ap\'ery numbers, Domb numbers and Almkvist--Zudilin numbers, respectively.

Finally, let $H(x)$ be defined by
$$
H(x)=\sum_{n=0}^\infty \biggl\{\sum_{j=0}^n {\binom{n}{j}}^4\biggr\}x^n.
$$

In our results below we use the labels ``Level 1'', ``Level 2'' etc., to distinguish the appearance
of different generating functions; the function $H(x)$ is labeled ``Level 10'' while the other labels can be extracted from the subscripts of
the corresponding functions. Levels themselves, in particular, their origins and meaning, are discussed further in Section~\ref{sec7} in the context of modular forms.

\section{Results: Part 1}
\label{sec3}

Our first meta-identity is the subject of the following theorem.

\begin{Theorem}
\label{t1}
The following forty-two functions are equal, in a neighborhood of $p=0$:
\begin{align}
\intertext{Level $1$\textup:}
\label{11}
\lefteqn{
\frac{1}{(1 +4p -8p^2 ) ^{1/2}(1 +228p-408p^2   -128p^3 -192p^4+768p^5-512p^6) ^{1/2}}
}  \\
&\quad\times
F_1\biggl(\frac{p( 1-p)^3 ( 1-4p) ^{12} ( 1-2p)( 1+2p)^3 }{ (1 +4p -8p^2 )^3 (1 +228p-408p^2   -128p^3 -192p^4+768p^5-512p^6)^3}\biggr)
\nonumber
\displaybreak[2]\\
\label{12}
&=\frac{1}{(1 -2p + 4p^2) ^{1/2}(1 -6p +240p^2-920p^3+960p^4-96p^5 +64p^6) ^{1/2}}
\\
&\quad\times
F_1\biggl(\frac {p^2 ( 1-p )^6 ( 1-4p )^6( 1-2p )^2 ( 1+2p )^6} { (1 -2p + 4p^2)^3 (1 -6p +240p^2-920p^3+960p^4-96p^5 +64p^6)^3} \biggr)
\nonumber
\displaybreak[2]\\
\label{13}
&=\frac{1}{(1+4p -8p^2 ) ^{1/2} (1 -12p +72p^2-128p^3-192p^4+768p^5-512p^6 ) ^{1/2}}
\\
&\quad\times
F_1\biggl(\frac {p^3 ( 1-p )  ( 1-4p )^4 ( 1-2p )^3 ( 1+2p ) } { (1+4p -8p^2 )^3 (1 -12p +72p^2-128p^3-192p^4+768p^5-512p^6 )^3}\biggr)
\nonumber
\displaybreak[2]\\
\label{14}
&=\frac{1}{(1-2p -2p^2 ) ^{1/2}(1-6p +6p^2+16p^3+204p^4-456p^5-8p^6) ^{1/2}}
\\
&\quad\times
F_1\biggl(\frac {p^4 ( 1-p ) ^{12} ( 1-4p )^3 ( 1-2p )  ( 1+2p )^3}{ (1-2p -2p^2 )^3(1-6p +6p^2+16p^3+204p^4-456p^5-8p^6)^3} \biggr)
\nonumber
\displaybreak[2]\\
\label{16}
&=\frac{1}{(1-2p+ 4p^2 ) ^{1/2}( 1-6p+40p^3-96p^5 +64p^6) ^{1/2}}
\\
&\quad\times
F_1\biggl(\frac {p^6 ( 1-p )^2 ( 1-4p )^2 ( 1-2p )^6 ( 1+2p )^2}{ (1-2p+ 4p^2 )^3( 1-6p+40p^3-96p^5 +64p^6)^3 } \biggr)
\nonumber
\displaybreak[2]\\
\label{112}
&=\frac{1}{(1-2p -2p^2 ) ^{1/2} (1-6p +6p^2+16p^3-36p^4+24p^5 -8p^6) ^{1/2}}
\\
&\quad\times
F_1\biggl(\frac {p^{12} ( 1-p )^4 ( 1-4p ) ( 1-2p )^3 ( 1+2p ) } { (1-2p -2p^2 )^3 (1-6p +6p^2+16p^3-36p^4+24p^5 -8p^6)^3} \biggr)
\nonumber
\displaybreak[2]\\
\label{1m1}
&=\frac{1}{(1 -8p +4p^2 ) ^{1/2} (1 -240p  +1932p^2 -5888p^3 +7728p^4-3840p^5+64p^6) ^{1/2}}
\\
&\quad\times
F_1\biggl(\frac {-p ( 1-p )^3 ( 1-4p )^3 ( 1-2p )^4  ( 1+2p ) ^{12}} { (1 -8p +4p^2 )^3 (1 -240p  +1932p^2 -5888p^3 +7728p^4-3840p^5+64p^6)^3}\biggr)
\nonumber
\displaybreak[2]\\
\label{1m3}
&=\frac{1}{(1 -8p +4p^2 ) ^{1/2} (1 +12p^2   -128p^3 +48p^4+64p^6) ^{1/2}}
\\
&\quad\times
F_1\biggl(\frac {-p^3 ( 1-p )  ( 1-4p )  ( 1-2p )^{12}  ( 1+2p )^4} { (1 -8p +4p^2 )^3 (1 +12p^2   -128p^3 +48p^4+64p^6)^3}\biggr)
\nonumber
\displaybreak[2]\\
\intertext{Level $2$\textup:}
\label{21}
&=\frac{1}{1+20p-48p^2+32p^3-32p^4}\,
F_2\biggl(\frac {p ( 1-p )^3 ( 1-4p )^6 ( 1-2p )  ( 1+2p )^3} { (1 +20p -48p^2 +32p^3-32p^4)^4}\biggr)
\\
\label{22}
&=\frac{1}{1-4p+24p^2-40p^3 -8p^4}\,
F_2\biggl(\frac {p^2 ( 1-p )^6 ( 1-4p )^3 ( 1-2p )  ( 1+2p )^3}{ ( 1-4p+24p^2 -40p^3  -8p^4)^4}\biggr)
\displaybreak[2]\\
\label{23}
&=\frac{1}{1-4p+32p^3-32p^4}\,
F_2\biggl(\frac {p^3 ( 1-p )  ( 1-4p )^2 ( 1-2p )^3 ( 1+2p ) } { (1 -4p +32p^3-32p^4 )^4}\biggr)
\displaybreak[2]\\
\label{26}
&=\frac{1}{1-4p+8p^3-8p^4}\,
F_2\biggl(\frac {p^6 ( 1-p )^2 ( 1-4p ) ( 1-2p )^3 ( 1+2p ) } { (1-4p +8p^3 -8p^4 )^4}\biggr)
\displaybreak[2]\\
\label{2m1}
&=\frac{1}{1-28p+96p^2-112p^3+16p^4}\,
F_2\biggl(-\frac {p ( 1-p )^3 ( 1-4p )^3 ( 1-2p )^2  ( 1+2p )^6} { (1 -28p +96p^2 -112p^3+16p^4)^4}\biggr)
\displaybreak[2]\\
\label{2m3}
&=\frac{1}{1-4p-16p^3+16p^4}\,
F_2\biggl(-\frac {p^3 ( 1-p )  ( 1-4p ) ( 1-2p )^6  ( 1+2p )^2} { (1 -4p -16p^3+16p^4)^4}\biggr)
\displaybreak[2]\\
\intertext{Level $3$\textup:}
\label{31}
&=\frac{1}{(1+4p-8p^2)^2}\,
F_3\biggl(\frac {p ( 1-p )  ( 1-4p )^4 ( 1-4p^2 ) }{ ( 1+4p-8p^2 )^6}\biggr)
\\
\label{32}
&=\frac{1}{(1-2p+4p^2)^2}\,
F_3\biggl(\frac {p^2 ( 1-p )^2 ( 1-4p )^2 ( 1-4p^2 )^2}{ ( 1-2p+4p^2 )^6}\biggr)
\displaybreak[2]\\
\label{34}
&=\frac{1}{(1-2p-2p^2)^2 }\,
F_3\biggl(\frac {p^4 ( 1-p )^4 ( 1-4p ) ( 1-4p^2 ) }{ ( 1-2p-2p^2 )^6}\biggr)
\displaybreak[2]\\
\label{3m}
&=\frac{1}{(1-8p+4p^2)^2 }\,
F_3\biggl(-\frac {p ( 1-p ) ( 1-4p ) ( 1-4p^2 )^4 }{ ( 1-8p+4p^2 )^6}\biggr)
\displaybreak[2]\\
\intertext{Level $4$\textup:}
\label{41}
&=\frac {1}{ ( 1-2p )  ( 1+2p )^3}\,
F_4\biggl(\frac {p ( 1-p )^3 ( 1-4p )^3}{ ( 1-2p )^2 ( 1+2p )^6}\biggr)
\\
\label{43}
&=\frac {1}{ ( 1-2p )^3 ( 1+2p ) }\,
F_4\biggl(\frac {p^3 ( 1-p )  ( 1-4p ) }{ ( 1-2p )^6 ( 1+2p )^2}\biggr)
\displaybreak[2]\\
\label{4m1}
&=\frac {1}{ ( 1-4p )^3 }\,
F_4\biggl(-\frac {p ( 1-2p )(1+2p)^3(1-p)^3 }{( 1-4p )^6}\biggr)
\displaybreak[2]\\
\label{4m3}
&=\frac {1}{ ( 1-4p ) }\,
F_4\biggl(-\frac {{p^3} ( 1-2p )^3(1+2p)(1-p)  }{( 1-4p )^2}\biggr)
\displaybreak[2]\\
\label{4m2}
&=\frac {1}{(1-2p)^{1/2}(1+2p)^{3/2} ( 1-4p ) ^{3/2} }\,
F_4\biggl(-\frac {p^2 ( 1-p )^6}{(1-2p)(1+2p)^3 ( 1-4p )^3}\biggr)
\displaybreak[2]\\
\label{4m6}
&=\frac {1}{(1-2p)^{3/2}(1+2p)^{1/2} ( 1-4p ) ^{1/2} }\,
F_4\biggl(-\frac{p^6( 1-p )^2}{(1-2p)^3(1+2p) ( 1-4p )}\biggr)
\displaybreak[2]\\
\intertext{Level $6$, functions $F$\textup:}
\label{Faa1}
&=\frac{1}{1-16p+24p^2+32p^3-32p^4}\,F_{6a}\biggl(\frac{p(1-p)(1-4p)^2(1-2p)(1+2p)}{(1-16p+24p^2+32p^3-32p^4)^2}\biggr)
\\
\label{Faa2}
&=\frac{1}{1-4p-12p^2+32p^3-8p^4}\,F_{6a}\biggl(\frac{p^2(1-p)^2(1-4p)(1-2p)(1+2p)}{(1-4p-12p^2+32p^3-8p^4)^2}\biggr)
\\
\label{Faa3}
&=\frac{1}{1+8p-48p^2+32p^3+16p^4}\,F_{6a}\biggl(\frac{-p(1-p)(1-4p)(1-2p)^2(1+2p)^2}{(1+8p-48p^2+32p^3+16p^4)^2}\biggr)
\displaybreak[2]\\
\label{Fbb1}
&=\frac{1}{(1-2p+4p^2)(1+4p-8p^2)}\,F_{6b}\biggl(\frac{p(1-p)(1-4p)^2(1-2p)(1+2p)}{(1-2p+4p^2)^2(1+4p-8p^2)^2}\biggr)
\\
\label{Fbb2}
&=\frac{1}{(1-2p-2p^2)(1-2p+4p^2)}\,F_{6b}\biggl(\frac{p^2(1-p)^2(1-4p)(1-2p)(1+2p)}{(1-2p-2p^2)^2(1-2p+4p^2)^2}\biggr)
\\
\label{Fbb3}
&=\frac{1}{(1-2p+4p^2)(1-8p+4p^2)}\,F_{6b}\biggl(\frac{-p(1-p)(1-4p)(1-2p)^2(1+2p)^2}{(1-2p+4p^2)^2(1-8p+4p^2)^2}\biggr)
\displaybreak[2]\\
\label{Fcc1}
&=\frac{1}{1+8p^2-32p^3+32p^4}\,F_{6c}\biggl(\frac{p(1-p)(1-4p)^2(1-2p)(1+2p)}{(1+8p^2-32p^3+32p^4)^2}\biggr)
\\
\label{Fcc2}
&=\frac{1}{1-4p+4p^2+8p^4}\,F_{6c}\biggl(\frac{p^2(1-p)^2(1-4p)(1-2p)(1+2p)}{(1-4p+4p^2+8p^4)^2}\biggr)
\\
\label{Faa3}
&=\frac{1}{1-8p+32p^2-32p^3+16p^4}\,F_{6c}\biggl(\frac{-p(1-p)(1-4p)(1-2p)^2(1+2p)^2}{(1-8p+32p^2-32p^3+16p^4)^2}\biggr)
\displaybreak[2]\\
\intertext{Level $6$, functions $G$\textup:}
\label{aa1}
&=\frac{1}{(1+2p)(1-p)}\,G_{6a}\biggl(\frac{p(1-4p)^2(1-2p)}{(1+2p)(1-p)}\biggr)
\\
\label{aa2}
&=\frac{1}{(1-2p)(1-p)^2}\,G_{6a}\biggl(\frac{p^2(1+2p)(1-4p)}{(1-2p)(1-p)^2}\biggr)
\\
\label{aa3}
&=\frac{1}{(1-p)(1-4p)(1-2p)^2}\,G_{6a}\biggl(\frac{-p(1+2p)^2}{(1-p)(1-4p)(1-2p)^2}\biggr)
\displaybreak[2]\\
\label{bb1}
&=\frac{1}{(1-4p)^2}\,G_{6b}\biggl(\frac{p(1+2p)(1-2p)(1-p)}{(1-4p)^2}\biggr)
\\
\label{bb2}
&=\frac{1}{(1-2p)(1+2p)(1-4p)}\,G_{6b}\biggl(\frac{p^2(1-p)^2}{(1-2p)(1+2p)(1-4p)}\biggr)
\\
\label{bb3}
&=\frac{1}{(1-2p)^2(1+2p)^2}\,G_{6b}\biggl(\frac{-p(1-p)(1-4p)}{(1-2p)^2(1+2p)^2}\biggr)
\displaybreak[2]\\
\label{cc1}
&=\frac{1}{(1-p)(1+2p)(1-4p)^2}\,G_{6c}\biggl(\frac{p(1-2p)}{(1-p)(1+2p)(1-4p)^2}\biggr)
\\
\label{cc2}
&=\frac{1}{(1+2p)(1-p)^2(1-4p)}\,G_{6c}\biggl(\frac{p^2(1-2p)}{(1+2p)(1-p)^2(1-4p)}\biggr)
\\
\label{cc3}
&=\frac{1}{(1-p)(1-4p)(1+2p)^2}\,G_{6c}\biggl(\frac{-p(1-2p)^2}{(1-p)(1-4p)(1+2p)^2}\biggr).
\end{align}
\end{Theorem}

\section{Results: Part 2}
\label{sec4}

The results in this section are consequences of the results in previous section by taking square roots.
For example, by \eqref{clauclau} we find that the ${}_3F_2$ hypergeometric function in \eqref{21}
is related to the ${}_2F_1$ function by
\begin{multline}
\label{claueg}
\pFq{3}{2}{\frac14,\frac12,\frac34}{1,1}
{256\,{\frac {p ( 1-p )^3 ( 1-4p )^6
 ( 1-2p )  ( 1+2p )^3}
 { (1 +20p -48p^2 +32p^3-32p^4)^4}}}
\\
= \biggl\{\pFq{2}{1}{\frac14,\frac34}{1}
{{\frac {64p ( 1-p )^3 ( 1-2p )  ( 1+2p )^3}
 { (1 +20p -48p^2 +32p^3-32p^4)^2}}}\biggr\}^2.
\end{multline}
The argument in the ${}_2F_1$ function is obtained by noting that the solution of
$$
4x(1-x) = 256\,\frac{p(1-p)^3(1-4p)^6(1-2p )( 1+2p )^3}{ (1 +20p -48p^2 +32p^3-32p^4)^4}
$$
that satisfies $x=0$ when $p=0$ is given by
\begin{equation}
\label{qe}
x = \frac {64p ( 1-p )^3 ( 1-2p )  ( 1+2p )^3}{ (1 +20p -48p^2 +32p^3-32p^4)^2}.
\end{equation}
By applying \eqref{claueg} and taking the square root of the expression in \eqref{21} we get
an expression that involves the ${}_2F_1$ function. This can be done, in principle, for all
of the hypergeometric expressions in \eqref{11}--\eqref{4m6}.
In a similar way, the identity \eqref{zs} can be
used to take square roots of the expressions in \eqref{Faa1}--\eqref{cc3}.
After taking square roots, the arguments
of the resulting functions are not always rational functions of~$p$ as they are for the example~\eqref{qe}, above;
sometimes they are
algebraic funcions of $p$ involving square roots that arise from solving quadratic equations.

In Theorem~\ref{t2} below, we list all of the functions obtained by taking square roots of the
expressions in \eqref{11}--\eqref{cc3} for which the arguments are rational functions of~$p$.
The identities have been numbered so that the
formula $(4.x)$ in Theorem~\ref{t2} below is the square root of the corresponding expression~$(3.x)$ in
Theorem~\ref{t1} above.

\setcounter{equation}{8}

\begin{Theorem}
\label{t2}
The following twenty-three functions are equal, in a neighborhood of $p=0$:
\begin{align}
\intertext{Level $2$\textup:}
\label{2F21}
\lefteqn{
\frac{1}{(1 +20p -48p^2 +32p^3-32p^4 )^{1/2}}\,
f_2\biggl(\frac{p ( 1-p )^3 ( 1-2p )  ( 1+2p )^3} { (1 +20p -48p^2 +32p^3-32p^4)^2}\biggr)
}
\\
\label{2F22}
&= \frac{1}{(1-4p+24p^2 -40p^3  -8p^4)^{1/2} }\,
f_2\biggl(\frac{p^2 ( 1-p )^6}{ ( 1-4p+24p^2 -40p^3  -8p^4)^2}\biggr)
\displaybreak[2]\\
\label{2F23}
&=\frac{1}{(1 -4p +32p^3-32p^4 )^{1/2} }\,
f_2\biggl(\frac{p^3 ( 1-p )  ( 1-2p )^3 ( 1+2p ) } { (1 -4p +32p^3-32p^4 )^2}\biggr)
\displaybreak[2]\\
\label{2F26}
&=\frac{1}{(1-4p +8p^3 -8p^4 )^{1/2} }\,
f_2\biggl(\frac {64p^6 ( 1-p )^2 } { (1-4p +8p^3 -8p^4 )^2}\biggr)
\displaybreak[2]\\
\label{2F2m1}
&=\frac{1}{(1 -28p +96p^2 - 112p^3+16p^4 )^{1/2}}\,
f_2\biggl(\frac {-p ( 1-p )^3 ( 1-4p )^3 } { (1 -28p +96p^2 -112p^3+16p^4)^2}\biggr)
\displaybreak[2]\\
\label{2F2m3}
&=\frac{1}{(1-4p-16p^3+16p^4)^{1/2}}\,
f_2\biggl(\frac{-p^3 ( 1-p )  ( 1-4p ) } { (1 -4p -16p^3+16p^4)^2}\biggr)
\displaybreak[2]\\
\intertext{Level $3$\textup:}
\label{2F31}
&=\frac{1}{(1+4p-8p^2)}\,f_3\biggl(\frac{p ( 1-2p )}{ ( 1+4p-8p^2 )^3}\biggr)
\\
\label{2F32}
&=\frac{1}{(1-2p+4p^2)}\,f_3\biggl(\frac{p^2 ( 1-2p )^2 }{( 1-2p+4p^2)^3}\biggr)
\displaybreak[2]\\
\label{2F34}
&=\frac{1}{(1-2p-2p^2)}\,f_3\biggl(\frac{p^4 ( 1-2p ) } { ( 1-2p-2p^2 )^3}\biggr)
\displaybreak[2]\\
\label{2F3m}
&=\frac{1}{(1-8p+4p^2)}\,f_3\biggl(\frac{-p( 1-2p )^4}{ ( 1-8p+4p^2 )^3}\biggr)
\displaybreak[2]\\
\intertext{Level $4$\textup:}
\label{2F41}
&=\frac{1}{(1-2p )^{1/2}  ( 1+2p ) ^{3/2}}\,f_4\biggl(\frac{p ( 1-p )^3 }{ ( 1-2p ) ( 1+2p )^3}\biggr)
\\
\label{2F43}
&=\frac {1}{ ( 1-2p ) ^{3/2} ( 1+2p )^{1/2} }\,f_4\biggl(\frac{p^3 ( 1-p )}{ ( 1-2p )^3 ( 1+2p )}\biggr)
\displaybreak[2]\\
\label{2F4m1}
&=\frac {1}{ ( 1-4p )^{3/2} }\,f_4\biggl(\frac{-p (1-p)^3 }{( 1-4p )^3}\biggr)
\displaybreak[2]\\
\label{2F4m3}
&=\frac {1}{ ( 1-4p )^{1/2} }\,f_4\biggl(\frac {-p^3(1-p)  }{( 1-4p )}\biggr)
\displaybreak[2]\\
\intertext{Level $6$\textup:}
\setcounter{equation}{24}
\label{a1}
&=\frac{1}{(1 -4p)^2}\,f_{6a}\biggl(\frac{p\,(1-2p)}{(1-4p)^2}\biggr)
\\
\label{a2}
&=\frac{1}{(1 -4p)(1+2p)}\,f_{6a}\biggl(\frac{p^2}{(1-4p)(1+2p)}\biggr)
\\
\label{a3}
&=\frac{1}{(1 + 2p)^2}\,f_{6a}\biggl(\frac{-p}{(1+2p)^2}\biggr)
\displaybreak[2]\\
\label{b1}
&=\frac{1}{(1 - p)(1+2p)}\,f_{6b}\biggl(\frac{p(1-2\, p)}{(1-p)(1+2p)}\biggr)
\\
\label{b2}
&=\frac{1}{(1 - p)^2}\,f_{6b}\biggl(\frac{p^2}{(1-p)^2}\biggr)
\\
\label{b3}
&=\frac{1}{(1 - p)(1-4p)}\,f_{6b}\biggl(\frac{-p}{(1-p)(1-4p)}\biggr)
\displaybreak[2]\\
\label{c1}
&=f_{6c}\bigl(p(1-2p)\bigr)
\\
\label{c2}
&=\frac{1}{1-2p}\,f_{6c}\biggl(\frac{p^2}{1-2p}\biggr)
\\
\label{c3}
&=\frac{1}{(1-2p)^2}\,f_{6c}\biggl( \frac{-p}{(1-2p)^2}\biggr).
\end{align}
\end{Theorem}

\section{Results: Part 3}
\label{sec5}

The results in this section involve transformations of degrees $2$, $5$ and $10$.

\begin{Theorem}
\label{t3}
The following thirteen functions are equal, in a neighborhood of $p=0$:
\begin{align}
\intertext{Level $1$\textup:}
\label{511}
\lefteqn{
\frac{1}{(1+236p +1440p^2+1920p^3+3840p^4+256p^5+256p^6)^{1/2}}
}
\\
&\quad\times
F_1\biggl(\frac {p ( 1-4p ) ^{10} ( 1+p )^5}{(1 +236p +1440p^2  +1920p^3 +3840p^4 +256p^5 + 256p^6 )^3}\biggr)
\nonumber
\displaybreak[2]\\
\label{512}
&=\frac{1}{(1-4p+240p^2-480p^3+1440p^4-944p^5+16p^6)^{1/2}}
\\
&\quad\times
F_1\biggl(\frac {p^2 ( 1-4p )^5 ( 1+p ) ^{10}}{ ( 1 -4p +240p^2 -480p^3 +1440p^4 -944p^5 + 16p^6)^3}\biggr)
\nonumber
\displaybreak[2]\\
\label{513}
&=\frac{1}{ (1 -4p +256p^5  + 256p^6) ^{1/2}}\,
F_1\biggl(\frac{p^5 ( 1-4p )^2 ( 1+p ) }{( 1 -4p +256p^5  + 256p^6 )^3}\biggr)
\displaybreak[2]\\
\label{514}
&=\frac{1}{ ( 1 -4p +16p^5 +16p^6  ) ^{1/2}}\,
F_1\biggl(\frac {p^{10} ( 1-4p )  ( 1+p )^2}{ ( 1 -4p +16p^5 +16p^6  )^3}\biggr)
\displaybreak[2]\\
\intertext{Level $2$\textup:}
\label{521}
&=\frac{1}{ ( 1+4p^2 ) ^{1/2} ( 1+22p-4p^2  ) }\,
F_2\biggl(\frac {p ( 1+p )^5 ( 1-4p )^5}{ ( 1+4p^2 )^2 ( 1+22p-4p^2  ) ^4}\biggr)
\\
\label{522}
&=\frac{1}{ ( 1+4p^2 )^{1/2} ( 1-2p-4p^2 )  }\,
F_2\biggl(\frac {p^5 ( 1+p )  ( 1-4p ) }{ ( 1+4p^2 )^2 ( 1-2p-4p^2 )^4}\biggr)
\displaybreak[2]\\
\intertext{Level $4$\textup:}
\label{541}
&=\frac{1}{  ( 1-4p ) ^{5/2} }\,
F_4\biggl(-p\biggl(\frac{  1+p }{ 1-4p} \biggr)^5\biggr)
\\
\label{542}
&=\frac{1}{ ( 1-4p )^{1/2} }\,
F_4\biggl(-p^5\biggl(\frac{  1+p }{ 1-4p} \biggr)\biggr)
\displaybreak[2]\\
\intertext{Level $5$, functions $F$\textup:}
\label{F5aa1}
&=\frac{1}{ ( 1 + 4p^2 )^{1/2} ( 1 + 4p + 8p^2 )}\,
F_5\biggl(\frac {p ( 1-4p )^2 ( 1+p ) }{ ( 1 + 4p^2 )  ( 1 + 4p + 8p^2 )^2 } \biggr)
\\
\label{F5aa2}
&=\frac{1}{( 1 + 4p^2 )^{1/2}( 1 -2p+2p^2 ) }  \,
F_5\biggl( \frac {p^2 ( 1-4p )  ( 1+p )^2}{( 1 + 4p^2 )( 1 -2p+2p^2 )^2 } \biggr)
\displaybreak[2]\\
\intertext{Level $5$, functions $G$\textup:}
\label{5aa1}
&=\frac{1}{( 1+p )  ( 1-4p )^2}  \,
G_5\biggl( \biggl(\frac p{  1+p} \biggr)  \biggl(\frac{1}{ 1-4p} \biggr)^2 \biggr)
\\
\label{5aa2}
&=\frac{1}{( 1+p )^2 ( 1-4p )} \,
G_5\biggl( \biggl(\frac p{  1+p} \biggr)^2  \biggl(\frac{1}{ 1-4p} \biggr)  \biggr)
\displaybreak[2]\\
\intertext{Level $10$, function $H$\textup:}
\label{10aa1}
&= \frac{1}{ ( 1+4p^2 ) ^{3/2}}  \,
H \biggl( \frac {p ( 1+p )  ( 1-4p ) }{ ( 1+4p^2 )^2} \biggr).
\end{align}
\end{Theorem}

\section{Proofs of Theorems \ref{t1}, \ref{t2} and \ref{t3}}
\label{sec6}

In this section we will provide proofs of Theorems \ref{t1}, \ref{t2} and \ref{t3}.
We begin with a proof of Theorem~\ref{t2} because it
is the simplest and therefore the explicit details are the easiest to write down.

\begin{proof}[Proof of Theorem \textup{\ref{t2}}.]
The hypergeometric function
$$
\pFq{2}{1}{s,1-s}{1}{x}
$$
satisfies the second order linear differential equation
$$
\frac{\d}{\d x}\biggl(x(1-x)\,\frac{\d z}{\d x}\biggr)=s(1-s)z.
$$
By changing variables, it can be shown that each function in \eqref{2F21}--\eqref{2F4m3}
satisfies the differential equation
\begin{equation}
\label{dzdp}
\frac{\d}{\d p}\biggl(p(1-p)(1-4p)(1-2p)(1+2p)\frac{\d y}{\d p}\biggr)
=2(1-4p)(1+4p-8p^2)y.
\end{equation}
In a similar way, the recurrence relation \eqref{recw1} implies that the each of the functions $f_{6a}$, $f_{6b}$ and $f_{6c}$ satisfies
a second order linear differential equation of the form
\begin{equation}
\label{zagde22}
\frac{\d}{\d x}\biggl(x(1-\alpha x-\gamma x^2)\frac{\d z}{\d x}\biggr)=(\beta+ \gamma x)z.
\end{equation}
On specializing the parameters and changing variables, we may deduce that each function \eqref{a1}--\eqref{c3} also satisfies \eqref{dzdp}.
By direct calculation, we find that the first two terms in the
expansions about $p=0$ of each function in \eqref{2F21}--\eqref{2F4m3}
or \eqref{a1}--\eqref{c3} are given by
$$
y=1+2p+O(p^2).
$$
It follows that the twenty-three functions \eqref{2F21}--\eqref{c3} are all equal.
\end{proof}

The proofs of Theorems \ref{t1} and \ref{t3} are similar and go as follows.

\begin{proof}[Proof of Theorems \textup{\ref{t1}} and \textup{\ref{t3}}.]
Each function in \eqref{11} -- \eqref{cc3}
(respectively, \eqref{511}--\eqref{10aa1}) satisfies the same
third order linear differential equation. Moreover, the first three terms in the expansion of each function
in \eqref{11} -- \eqref{cc3}
(respectively, \eqref{511}--\eqref{10aa1}) about $p=0$ are
given by
$$
y=1 + 4p + 16p^2 +O(p^3) \qquad \text{(respectively, $y=1 + 2p + 6p^2 +O(p^3)$).}
$$
It follows that the forty-two functions in \eqref{11}--\eqref{cc3}
(respectively, the thirteen functions in \eqref{511}--\eqref{10aa1}) are equal.
\end{proof}

It may be emphasized that the proofs outlined above are conceptually simple, establish the correctness of the results,
and require very little mathematical knowledge.
However, the proofs above are not illuminating: they give no insight into the structure of the identities, nor any reason
for why the results exist or any clue as to how the identities were discovered.
In the next section, the theory of modular forms will be used to give alternative proofs that also provide an
explanation for how the identities were discovered.

In practice, the easiest way to show that each of the functions in \eqref{2F21}--\eqref{c3}
satisfies the differential equation \eqref{dzdp} is to expand each function as a power series about $p=0$ to a large number
of terms, and then use computer algebra to determine the differential equation. For example using Maple,
the differential equation satisfied by the function in~\eqref{2F21} can be determined from the first $50$ terms in the
series expansion in powers of~$p$, by entering the commands

\begin{verbatim}
> with(gfun):
> n := 64*p*(1-p)^3*(1-2*p)*(1+2*p)^3;
> d := (-32*p^4+32*p^3-48*p^2+20*p+1)^2;
> s := series(d^(-1/4)*hypergeom([1/4, 3/4], [1], n/d), p, 50):
> L := seriestolist(s):
> guesseqn(L, y(p));
\end{verbatim}

\noindent
This produces the output
\begin{multline*}
\Big[\Big\{  \left( -16\,{p}^{5}+20\,{p}^{4}-5\,{p}^{2}+p \right) {
\frac {{\rm d}^{2}}{{\rm d}{p}^{2}}}y \left( p \right) + \left( -80\,{
p}^{4}+80\,{p}^{3}-10\,p+1 \right) {\frac {\rm d}{{\rm d}p}}y \left( p
 \right)
 \\
 + \left( -64\,{p}^{3}+48\,{p}^{2}-2 \right) y \left( p
 \right) ,y \left( 0 \right) =1,\mbox {D} \left( y \right)  \left( 0
 \right) =2 \Big\}\Big] ,{\it ogf}
\end{multline*}
which is equivalent to~\eqref{dzdp}.

\section{Modular origins}
\label{sec7}

We will now explain how the identities in Theorems~\ref{t1}, \ref{t2} and \ref{t3} were found.
The explanation also puts the formulas into context and reveals why they exist.

A modular explanation for Theorem~\ref{t1} requires the theory of modular forms for level~$12$ as developed in~\cite{cooperye}.
Some of the details may be summarized as follows.
Let $h$, $p$ and $z$ be defined by
\begin{gather*}
h=q\prod_{j=1}^\infty \frac{(1-q^{12j-11})(1-q^{12j-1})}{(1-q^{12j-7})(1-q^{12j-5})},
\\
p=\frac{h}{1+h^2}
\quad\text{and}\quad
z=q\,\frac{\d}{\d q}\, \log h.
\end{gather*}
Ramanujan's Eisenstein series $P$ and $Q$ are defined by
$$
P(q)=1-24\sum_{j=1}^\infty \frac{jq^j}{1-q^j}
\quad\text{and}\quad
Q(q)=1+240\sum_{j=1}^\infty \frac{j^3q^j}{1-q^j}.
$$
Dedekind's eta-function is defined for $\Im\tau>0$ and $q=\exp(2\pi i\tau)$ by
$$
\eta(\tau)=q^{1/24}\prod_{j=1}^\infty (1-q^j).
$$

We will require the following three lemmas, extracted from the literature.

\begin{Lemma}
\label{7.1}
Suppose $k$ and $m$ are positive divisors of $12$ and $k>m$.
Then there are rational functions $r_{k,m}(h)$ such that
$$
kP(q^k)-mP(q^m)=z \times r_{k,m}(h).
$$
Furthermore, there are rational functions $s_m(h)$ and $t_m(h)$ such that
$$
Q(q^m)=z^2 \times s_m(h)
\quad\text{and}\quad
\eta^{24}(m\tau)= z^6 \times t_m(h).
$$
\end{Lemma}

\begin{proof}
The result for $P$ follows from \cite[Theorem 4.5]{cooperye}; the result for~$Q$ is
Theorem~4.6 in \cite{cooperye}; and the result for Dedekind's eta-function is Theorem~4.2 in \cite{cooperye}.
\end{proof}

Explicit formulas for the rational functions $r_{k,m}(h)$, $s_m(h)$ and $t_m(h)$ can be determined from
the information in~\cite{cooperye}. Some examples will be given below, after Lemma~\ref{7.3}.
\begin{Lemma}
\label{7.2}
For $\ell\in\{1,2,3,4\}$, let $Z_\ell$ be defined by
\begin{equation}
\label{zall4def2}
Z_\ell(q) = \begin{cases}
Q^{1/2}(q) & \text{if $\ell=1$,} \\[1.5mm]
\displaystyle{\frac{\ell P(q^\ell)-P(q)}{\ell-1}} & \text{if $\ell=2$, $3$ or $4$.}
\end{cases}
\end{equation}
Then the following parameterizations of hypergeometric functions hold:
\begin{gather*}
Z_1(q)=F_1\biggl(\biggl(\frac{\eta^4(\tau)}{Z_1(q)}\biggr)^6\biggr),
\quad
Z_2(q)=F_2\biggl(\biggl(\frac{\eta^2(\tau)\eta^2(2\tau)}{Z_2(q)}\biggr)^4\biggr),
\\
Z_3(q)=F_3\biggl(\biggl(\frac{\eta^2(\tau)\eta^2(3\tau)}{Z_3(q)}\biggr)^3\biggr)
\quad\text{and}\quad
Z_4(q)=F_4\biggl(\biggl(\frac{\eta^4(\tau)\eta^4(4\tau)}{\eta^4(2\tau)\,Z_4(q)}\biggr)^2\biggr),
\end{gather*}
where, as usual, $q=\exp(2\pi i\tau)$ and the hypergeometric functions are as in~\eqref{bhf}.
\end{Lemma}

\begin{proof}
These results may be found in \cite{bbg}, \cite{domb} or \cite{cooperlms}, or they
can be proved by putting together identities in those references and applying the special case of Clausen's identity
given by~\eqref{clauclau}.
\end{proof}

The parameter $\ell$ in Lemma~\ref{7.2} is called the level. The next lemma gives the analogous results for
level~$6$.

\begin{Lemma}
\label{7.3}
The following parameterizations hold:
\begin{align*}
\frac{\eta^6(\tau)\eta(6\tau)}{\eta^3(2\tau)\eta^2(3\tau)}
&=f_{6a}\biggl(\frac{\eta(2\tau)\eta^5(6\tau)}{\eta^5(\tau)\eta(3\tau)}\biggr),
\\
\frac{\eta^6(2\tau)\eta(3\tau)}{\eta^3(\tau)\eta^2(6\tau)}
&=f_{6b}\biggl(\frac{\eta^4(\tau)\eta^8(6\tau)}{\eta^8(2\tau)\eta^4(3\tau)}\biggr)
\intertext{and}
\frac{\eta(2\tau)\eta^6(3\tau)}{\eta^2(\tau)\eta^3(3\tau)}
&=f_{6c}\biggl(\frac{\eta^3(\tau)\eta^9(6\tau)}{\eta^3(2\tau)\eta^9(3\tau)}\biggr).
\end{align*}
\end{Lemma}

\begin{proof}
This is explained in Zagier's work \cite{zagier}. Our functions $f_{6b}$ and $f_{6c}$ correspond to the functions~$f(z)$
in \cite{zagier} in the cases C and A, respectively. The function $f_{6a}$ corresponds to the function~$f(z)$
in \cite{zagier} in case~F but with~$-q$ in place of~$q$.
\end{proof}

We are now ready to explain how the forty-two functions in Theorem~\ref{t1} arise, and why they are equal.
We will only focus on the level 3 functions in \eqref{31}--\eqref{3m} as an illustration; the other functions can be
obtained by a similar procedure by working with the other levels.

\begin{proof}[Proof of Theorem~\textup{\ref{t1}} using modular forms.]
By Lemma~\ref{7.1} and the explicit formulas in \cite[Theorem 4.2]{cooperye} we find that
$$
\eta^6(\tau)\eta^6(3\tau) = z^3 \times \frac{h(1+h^2)(1-h+h^2)}{(1-h^2)(1-h+h^2)^2}
$$
and
$$
\frac12\bigl(3P(q^3)-P(q)\bigr)
=z \times \frac{(1+4h-6h^2+4h^3+h^4)^2}{(1+h^2)(1-h+h^2)(1-4h+h^2)(1-h^2)}.
$$
Substituting these in the result for $Z_3(q)$ in Lemma~\ref{7.2} gives
\begin{multline*}
z \times \frac{(1+4h-6h^2+4h^3+h^4)^2}{(1+h^2)(1-h+h^2)(1-4h+h^2)(1-h^2)} \\
= \pFq{3}{2}{\frac13,\frac12,\frac23}{1,1}{108\frac{h(1+h^2)^4(1-h+h^2)(1-4h+h^2)^4(1-h^2)^2}{(1+4h-6h^2+4h^3+h^4)^6}}.
\end{multline*}
Under the change of variables $p=h/(1+h^2)$, this be written in the form
\begin{multline}
\label{nearly1}
\frac{z}{(1-p)(1-4p)\sqrt{1-4p^2}}\\
=\frac{1}{(1+4p-8p^2)^2}\;
\pFq{3}{2}{\frac13,\frac12,\frac23}{1,1}
{108\,\frac {p ( 1-p )  ( 1-4p )^4 ( 1-4p^2) }
{(1+4p-8p^2)^6}}.
\end{multline}
We will now obtain three further formulas akin to \eqref{nearly1} by replacing $q$ with
$q^2$, $q^4$ or $-q$ in the identity for $Z_3(q)$ in Lemma~\ref{7.2}; that is, by replacing $\tau$ with
$2\tau$, $4\tau$ and $\tau + \frac12$, respectively. First,
replacing $\tau$ with $2\tau$
in the formula for $Z_3$ in Lemma~\ref{7.2} and
using the parameterizations
$$
\eta^6(2\tau)\eta^6(6\tau) = z^3 \times \frac{h^2(1-h^2)}{(1+h^2)(1-h+h^2)(1-4h+h^2)}
$$
and
$$
 \frac12\bigl(3P(q^6)-P(q^2)\bigr)
=z \times \frac{(1-2h+6h^2-h^3+h^4)^2}{(1+h^2)(1-h+h^2)(1-4h+h^2)(1-h^2)}
$$
leads to the identity
\begin{multline}
\label{nearly2}
\frac{z}{(1-p)(1-4p)\sqrt{1-4p^2}}\\
= \frac{1}{(1-2p+4p^2)^2}\,
\pFq{3}{2}{\frac13,\frac12,\frac23}{1,1}
{108\,\frac {p^2 (1-p)^2 (1-4p)^2 (1-4p^2)^2}{(1-2p+4p^2)^6}}.
\end{multline}
Similarly, replacing $\tau$ with $4\tau$
in the formula for $Z_3$ in Lemma~\ref{7.2} and using the parameterizations
$$
\eta^6(4\tau)\eta^6(12\tau) = z^3 \times \frac{h^4(1-h+h^2)}{(1+h^2)^2(1-h^2)(1-4h+h^2)^2}
$$
and
$$
 \frac12\bigl(3P(q^{12})-P(q^4)\bigr)
= z \times \frac{(1-2h-2h^3+h^4)^2}{(1+h^2)(1-h+h^2)(1-4h+h^2)(1-h^2)}
$$
leads to the identity
\begin{multline}
\label{nearly3}
\frac{z}{(1-p)(1-4p)\sqrt{1-4p^2}}\\
= \frac{1}{(1-2p-2p^2)^2 }\,
\pFq{3}{2}{\frac13,\frac12,\frac23}{1,1}
{108\,\frac {p^4 (1-p)^4 (1-4p) (1-4p^2)}{(1-2p-2p^2)^6}}.
\end{multline}
In order to replace $q$ with $-q$, equivalently $\tau$ with $\tau+\frac12$, use the identity
$$
\prod_{j=1}^\infty (1-(-q)^j) = \prod_{j=1}^\infty \frac{(1-q^{2j})^3}{(1-q^j)(1-q^{4j})}
$$
to get
\begin{equation}
\label{minus1}
\eta^6(\tau)\eta^6(3\tau)\bigg|_{\tau \mapsto \tau+\frac12}
=\frac{-\eta^{18}(2\tau)\eta^{18}(2\tau)}{\eta^6(\tau)\eta^6(3\tau)\eta^6(4\tau)\eta^6(12\tau)},
\end{equation}
and use the identity
\begin{equation}
\label{pminus}
P(-q) = -P(q)+6P(q^2)-4P(q^4)
\end{equation}
to deduce that
\begin{multline}
\label{minus2}
\lefteqn{ \frac12\bigl(3P(-q^3)-P(-q)\bigr)} \\
 = -\frac12\bigl(3P(q^3)-P(q)\bigr) + 3\bigl(3P(q^6)-P(q^2)\bigr)-2\bigl(3P(q^{12})-P(q^4)\bigr).
 \end{multline}
Then, \eqref{minus1} and \eqref{minus2} can be used along with the parameterizations given above, to
replace $q$ with $-q$ in the formula for $Z_3$ in Lemma~\ref{7.2} and produce the identity
\begin{multline}
\label{nearly4}
\frac{z}{(1-p)(1-4p)\sqrt{1-4p^2}}\\
=  \frac{1}{(1-8p+4p^2)^2 }\,
\pFq{3}{2}{\frac13,\frac12,\frac23}{1,1}
{-108\,\frac{p (1-p) (1-4p) (1-4p^2)^4}{(1-8p+4p^2)^6}}.
\end{multline}
Finally, equating \eqref{nearly1}, \eqref{nearly2}, \eqref{nearly3} and \eqref{nearly4}
shows the equality of \eqref{31}--\eqref{3m} in Theorem~\ref{t1}.

The identities \eqref{11}--\eqref{1m3}, \eqref{21}--\eqref{2m3} and \eqref{41}--\eqref{4m6} can be
obtained in the same way by using the results for $Z_1$, $Z_2$ and $Z_4$, respectively, in Lemma~\ref{7.2}.
The identities \eqref{Faa1}--\eqref{cc3} can be obtained using the parameterizations in Lemma~\ref{7.3}
together with the identity~\eqref{zs}.
\end{proof}

\begin{Corollary}
Suppose that $y$ is the solution of the differential equation
$$
\frac{\d}{\d p} \biggl(p(1-p)(1-4p)(1-2p)(1+2p) \frac{\d y}{\d p}\biggr)
=2(1-4p)(1+4p-8p^2)y
$$
that satisfies the initial conditions $y(0)=1$, $y'(0)=2$, and suppose
$$
y=\sum_{n=0}^\infty t(n)p^n
$$
in a neighborhood of $p=0$. Then
\begin{multline*}
(n+1)^2 t(n+1) \\
= (5n^2+5n+2)t(n) -4(5n^2-5n+2)t(n-2) + 16(n-1)^2 t(n-3),
\end{multline*}
the coefficient of the term $t(n-1)$ being zero.
Furthermore, $y$ and $p$ may be parameterized by the modular forms
$$
y = \frac{\eta(2\tau)\eta^6(3\tau)}{\eta^2(\tau)\eta^3(6\tau)}
\quad\text{and}\quad
p = \frac{\eta(\tau)\eta^3(12\tau)}{\eta^3(3\tau)\eta(4\tau)}.
$$
\end{Corollary}

\begin{proof}
The recurrence relation for the coefficients can be deduced immediately by substituting the series
expansion into the differential equations. This is a routine procedure so we omit the details.

Next, let
$$
Y=\frac{z}{(1-p)(1-4p)\sqrt{1-4p^2}}.
$$
By the `proof of Theorem~\ref{t1} using modular forms' detailed above, and
especially \eqref{nearly1}, \eqref{nearly2}, \eqref{nearly3} and \eqref{nearly4},
all forty-two functions in Theorem~\ref{t1} are different expressions for $Y$. By the
change of variable $p=h/(1+h^2)$ and the formulas in~\cite{cooperye}, we have
$$
Y = z \times \frac{(1+h^2)^3}{(1-h^2)(1-h+h^2)(1-4h+h^2)} = \frac{\eta^2(2\tau)\eta^{12}(3\tau)}{\eta^4(\tau)\eta^6(6\tau)}.
$$
By Clausen's formula, the functions $y$ in Theorem~\ref{t2} are related to the functions $Y$ in
Theorem~\ref{t1} by $Y=y^2$, and it follows that
$$
y = \frac{\eta(2\tau)\eta^6(3\tau)}{\eta^2(\tau)\eta^3(6\tau)}.
$$
Moreover, by the formulas in~\cite{cooperye} we have
$$
p = \frac{h}{1+h^2} = \frac{\eta(\tau)\eta^3(12\tau)}{\eta^3(3\tau)\eta(4\tau)}.
$$
Finally, $y$ satisfies the required differential equation with respect to $p$, by~\eqref{dzdp}.
\end{proof}

Theorem~\ref{t3} can be proved in a similar way using the level $10$ function
$$
k=q\prod_{j=1}^\infty \frac{(1-q^{10j-9})(1-q^{10j-8})(1-q^{10j-2})(1-q^{10j-1})}
{(1-q^{10j-7})(1-q^{10j-6})(1-q^{10j-4})(1-q^{10j-3})}
$$
and letting
\begin{equation}
\label{pkk}
p=\frac{k}{1-k^2},
\end{equation}
and then using
the properties developed in \cite{cooper10} and \cite{cooper10a}, along with the results for level 5 modular forms
that are summarized in~\cite{cc}. We omit most of the details as they are similar to the proof of Theorem~\ref{t1} given
above. It is worth recording the modular parameterization.

\begin{Theorem}
Let the common power series of each expression in \eqref{511}--\eqref{10aa1} be denoted by
$$
\sum_{n=0}^\infty b(n)p^n.
$$
Then
$$
\frac{\eta(2\tau)\eta^{10}(5\tau)}{\eta^2(\tau)\eta^5(10\tau)}
=\sum_{n=0}^\infty b(n) \biggl(\frac{\eta(\tau)\eta^5(10\tau)}{\eta(2\tau)\eta^5(5\tau)}\biggr)^n.
$$
\end{Theorem}

\begin{proof}
By \eqref{pkk} and \cite[Theorem 3.5]{cooper10} we have
\begin{equation}
\label{pk2}
p=\frac{k}{1-k^2} = \frac{\eta(\tau)\eta^5(10\tau)}{\eta(2\tau)\eta^5(5\tau)}.
\end{equation}
Next, starting with \eqref{521} and using the formula for $Z_2$ in Lemma~\ref{7.2}, we get
\begin{align*}
\sum_{n=0}^\infty b(n)p^n
&= \frac{1}{(1+4p^2) ^{1/2} (1+22p-4p^2)} \\
& \qquad \times \pFq{3}{2}{\frac14,\frac12,\frac34}{1,1}
{256\,\frac{p (1+p)^5 (1-4p)^5}{(1+4p^2)^2 (1+22p-4p^2)^4}}
\displaybreak[2]\\
&= \biggl(\frac{1}{p(1+p)^5(1-4p)^5}\biggr)^{1/4} \times \biggl(\frac{\eta^2(\tau)\eta^2(2\tau)}{Z_2(q)}\biggr)\times Z_2(q)
\\
&= \biggl(\frac{1}{p(1+p)^5(1-4p)^5}\biggr)^{1/4} \times \eta^2(\tau)\eta^2(2\tau).
\end{align*}
Now use \eqref{pkk} to write $p$ in terms of $k$, and then use \cite[Theorem 3.5]{cooper10} to express
the resulting rational function of $k$ in terms of eta-functions, to get
\begin{equation}
\label{pk3}
\sum_{n=0}^\infty b(n)p^n  = \frac{\eta^{10}(5\tau)}{\eta^4(\tau)\eta(2\tau)\eta^5(10\tau)} \times \eta^2(\tau)\eta^2(2\tau)
= \frac{\eta(2\tau)\eta^{10}(5\tau)}{\eta^2(\tau)\eta^5(10\tau)}.
\end{equation}
The proof may be completed by substituting the result of \eqref{pk2} into \eqref{pk3}.
\end{proof}

\section{Special cases}
\label{sec8}

Many of the transformation formulas in Theorems~\ref{t1}, \ref{t2} and~\ref{t3} can be simplified, sometimes
significantly, by changing variables. We give several examples.

\begin{Example}
This example is from a paper by N.D.~Baruah and B.C.~Berndt \cite{baruah} and
the book by J.M.~Borwein and P.B.~Borwein \cite[pp. 180--181]{agm}.
Let $X=4x(1-x)$. Then
\begin{align*}
\pFq{3}{2}{\frac12,\frac12,\frac12}{1,1}{X}
&= \frac{1}{1-x}\, \pFq{3}{2}{\frac12,\frac12,\frac12}{1,1}{\frac{-4x}{(1-x)^2}}
\\
&= \frac{1}{\sqrt{1-x}}\,\pFq{3}{2}{\frac12,\frac12,\frac12}{1,1}{\frac{-x^2}{4(1-x)}}
\displaybreak[2]\\
&=\frac{1}{1+x}\, \pFq{3}{2}{\frac14,\frac12,\frac34}{1,1}{\frac{16x(1-x)^2}{(1+x)^4}}
\displaybreak[2]\\
&=\frac{1}{1-2x}\, \pFq{3}{2}{\frac14,\frac12,\frac34}{1,1}{\frac{-16x(1-x)}{(1-2x)^4}}
\displaybreak[2]\\
&= \frac{2}{\sqrt{4-X}}\,\pFq{3}{2}{\frac16,\frac12,\frac56}{1,1}{\frac{27X^2}{(4-X)^3}}
\\
&= \frac{1}{\sqrt{1-4X}}\,\pFq{3}{2}{\frac16,\frac12,\frac56}{1,1}{\frac{-27X}{(1-4X)^3}}.
\end{align*}
\end{Example}

\begin{proof}
Let
$$
x=\frac {16p (1-p)^3}{(1-2p) (1+2p)^3}
$$
in each of \eqref{41}, \eqref{4m1}, \eqref{4m2}, \eqref{21}, \eqref{2m1}, \eqref{12} and \eqref{1m1},
respectively.
\end{proof}

\begin{Example}
This example was considered by J.~Guillera and W.~Zudilin \cite[Eq.~(20)]{translation}:
\begin{align*}
\pFq{3}{2}{\frac13,\frac12,\frac23}{1,1}{4x(1-x)}
&= \frac{1}{\sqrt{1+8x}}\,
\pFq{3}{2}{\frac16,\frac12,\frac56}{1,1}{\frac{64x(1-x)^3}{(1+8x)^3}} \\
&= \frac{3}{\sqrt{9-8x}}\,
\pFq{3}{2}{\frac16,\frac12,\frac56}{1,1}{\frac{64x^3(1-x)}{(9-8x)^3}}.
\end{align*}
\end{Example}

\begin{proof}
Let
$$
x=\frac{27p^2 ( 1-2p )^2 }{( 1-2p+4p^2)^3}
$$
in~\eqref{32}, \eqref{12} and \eqref{16}, respectively.
\end{proof}

\begin{Example}
The identities in this example are from the unorganized pages of Ramanujan's second notebook~\cite[p.~258]{notebooks}:
$$
\sqrt{1+2x}\, \pFq{2}{1}{\frac13,\frac23}{1}{\frac{27x^2(1+x)^2}{4(1+x+x^2)^3}}
= (1+x+x^2)\,\pFq{2}{1}{\frac12,\frac12}{1}{\frac{x^3(2+x)}{1+2x}}
$$
and
\begin{multline*}
(2+2x-x^2)\,\pFq{2}{1}{\frac13,\frac23}{1}{\frac{27x(1+x)^4}{2(1+4x+x^2)^3}}
\\
= 2(1+4x+x^2)\, \pFq{2}{1}{\frac13,\frac23}{1}{\frac{27x^4(1+x)}{2(2+2x-x^2)^3}}.
\end{multline*}
Proofs, by a different method, have been given by Berndt et al.~\cite[Theorems 5.6 and~6.4]{bbg}.
\end{Example}

\begin{proof}
Take $p=-x/2$ in~\eqref{2F32} and~\eqref{2F43} to obtain
\begin{multline*}
\frac{1}{(1+x+x^2)}\,\pFq{2}{1}{\frac13,\frac23}{1}{\frac{27x^2 ( 1+x )^2 }{4( 1+x+x^2)^3}}
\\
=\frac {1}{ ( 1+x ) ^{3/2} ( 1-x )^{1/2} }\,\pFq{2}{1}{\frac12,\frac12}{1}{\frac{-x^3(2+x)}{(1+x)^3(1-x)}}.
\end{multline*}
Now apply Pfaff's transformation \cite[Theorem 2.2.5]{aar}
$$
\pFq{2}{1}{a,b}{c}{z} = (1-z)^{-a}\,\pFq{2}{1}{a,c-b}{c}{\frac{-z}{1-z}}
$$
to the right-hand side to obtain the first identity.

The second identity is obtained simply by putting $p=-x/2$ in~\eqref{2F34} and~\eqref{2F3m}.
\end{proof}

\begin{Example}
This example was studied by M.~Rogers \cite[Theorem 3.1]{rogers} and by
H.H.~Chan and W.~Zudilin \cite[Theorems 3.2 and 4.2]{mathematika}:
\begin{align*}
G_{6b}(x)
&= \frac{1}{1+16x}\, \pFq{3}{2}{\frac13,\frac12,\frac23}{1,1}{\frac{108x}{(1+16x)^3}}   \\
&= \frac{1}{1+4x}\, \pFq{3}{2}{\frac13,\frac12,\frac23}{1,1}{\frac{108x^2}{(1+4x)^3}}
\intertext{and}
G_{6c}(x)
&= \frac{1}{1+27x}\, \pFq{3}{2}{\frac14,\frac12,\frac34}{1,1}{\frac{256x}{(1+27x)^4}}  \\
&= \frac{1}{1+3x}\, \pFq{3}{2}{\frac14,\frac12,\frac34}{1,1}{\frac{256x^3}{(1+3x)^4}}.
\end{align*}
\end{Example}

\begin{proof}
The first group of identities is obtained by taking
$$
x=\frac{p(1-p)(1+2p)(1-2p)}{(1-4p)^2}
$$
in \eqref{31}, \eqref{32} and \eqref{bb1}.
The second group of identities may be proved by taking
$$
x=\frac{p(1-2p)}{(1-p)(1+2p)(1-4p)^2}
$$
in \eqref{21}, \eqref{23} and \eqref{cc1}.
\end{proof}

\begin{Example}
\label{cz6}
This example shows that the result of H.H.~Chan and W.~Zudilin \cite[Theorem 2.2]{mathematika} is
subsumed by Theorem~\ref{t1}:
\begin{align*}
\frac{1}{1+x}\,G_{6a}\biggl(\frac{x(1-8x)}{1+x}\biggr)
&=\frac{1}{1-8x}\,G_{6b}\biggl(\frac{x(1+x)}{1-8x}\biggr),
\\
\frac{1}{1-x}\,G_{6a}\biggl(\frac{x(1-9x)}{1-x}\biggr)
&=\frac{1}{1-9x}\,G_{6c}\biggl(\frac{x(1-x)}{1-9x}\biggr)
\\ \intertext{and}
\frac{1}{1+8x}\,G_{6b}\biggl(\frac{x(1+9x)}{1+8x}\biggr)
&=\frac{1}{1+9x}\,G_{6c}\biggl(\frac{x(1+8x)}{1+9x}\biggr).
\end{align*}
\end{Example}

\begin{proof}
The three identities may be obtained by setting
$$
x=p(1-2p), \quad x=\frac{p(1-2p)}{(1+2p)(1-p)}\quad\text{and}\quad x=\frac{p(1-2p)}{(1-4p)^2}
$$
respectively, in \eqref{aa1}, \eqref{bb1} and \eqref{cc1}.
\end{proof}

Alternative proofs of the identities in Example~\ref{cz6} may be given by taking
$$
x=\frac{p^2}{1-2p},\quad x=\frac{p^2}{(1-p)^2}\quad\text{and}\quad x=\frac{p^2}{1-2p-8p^2}
$$
in \eqref{aa2}, \eqref{bb2} and \eqref{cc2}, or by taking
$$
x=\frac{-p}{(1-2p)^2},\quad x=\frac{-p}{(1-p)(1-4p)} \quad\text{and}\quad x=\frac{-p}{(1+2p)^2}
$$
in \eqref{aa3}, \eqref{bb3} and \eqref{cc3}.

\begin{Example}
\label{6weight1}
Part of this example was mentioned in the identity \eqref{pb} as part of the introduction:
\begin{align*}
f_{6a}(x)
&=\frac{1}{1+9x}\,f_{6b}\biggl(\frac{x}{1+9x}\biggr)
=\frac{1}{1+8x}\,f_{6c}\biggl(\frac{x}{1+8x}\biggr),
\\
f_{6b}(x)
&=\frac{1}{1-9x}\,f_{6a}\biggl(\frac{x}{1-9x}\biggr)
=\frac{1}{1-x}\,f_{6c}\biggl(\frac{x}{1-x}\biggr),
\\
f_{6c}(x)
&=\frac{1}{1-8x}\,f_{6a}\biggl(\frac{x}{1-8x}\biggr)
=\frac{1}{1+x}\,f_{6b}\biggl(\frac{x}{1+x}\biggr).
\end{align*}
\end{Example}

\begin{proof}
Each of the three sets of identities can be proved by taking
$$
x=\frac{p(1-2p)}{(1-4p)^2},\quad
x=\frac{p(1-2p)}{(1-p)(1+2p)}\quad\text{and}\quad
x=p(1-2p),
$$
respectively, in \eqref{a1}, \eqref{b1} and \eqref{c1}.
\end{proof}

\begin{Example}
Here we give representations for the
functions in the previous example in terms of hypergeometric functions:
\begin{align*}
f_{6b}(x)
&=\frac{1}{\sqrt{1+18x-27x^2}}\,\pFq{2}{1}{\frac14,\frac34}{1}{\frac{64x}{(1+18x-27x^2)^2}}
\\
&=\frac{1}{\sqrt{1-6x-3x^2}}\,\pFq{2}{1}{\frac14,\frac34}{1}{\frac{64x^3}{(1-6x-3x^2)^2}}
\intertext{and}
f_{6c}(x)
&=\frac{1}{1+4x}\,\pFq{2}{1}{\frac13,\frac23}{1}{\frac{27x}{(1+4x)^3}}
\\
&=\frac{1}{1-2x}\,\pFq{2}{1}{\frac13,\frac23}{1}{\frac{27x^2}{(1-2x)^3}}.
\end{align*}
\end{Example}

\begin{proof}
The first set of identities may be proved by taking
$$
x=\frac{p(1-2p)}{(1-p)(1+2p)}
$$
in \eqref{2F21}, \eqref{2F23} and \eqref{b1}.
To prove the second set of identities, take $x=\nobreak p(1-2p)$
in \eqref{2F31}, \eqref{2F32} and \eqref{c1}.
\end{proof}

Just as for Example \ref{cz6}, the identities in Example \ref{6weight1} can be given alternative proofs
using \eqref{a2}, \eqref{b2} and \eqref{c2}, or by using \eqref{a3}, \eqref{b3} and \eqref{c3}.

\section{Applications}
\label{sec9}

In this section, we will show how the transformation formulas in Theorem~\ref{t1} can be used to establish
the equivalence of several of Ramanujan's series for $1/\pi$.
In the remainder of this section, we will use the binomial representation~\eqref{bhf}
of the related hypergeometric functions, so that the resulting formulas will be consistent with
the data in~\cite[Tables 3--6]{cc}.

\begin{Theorem}
\label{aycock1}
The following series identities are equivalent in the sense that any one can be obtained from the others by
using the transformation formulas in Theorem~\textup{\ref{t1}}:
\begin{equation}
\label{pi1}
\sum_{n=0}^\infty \binom{6n}{3n}\binom{3n}{n}\binom{2n}{n}\biggl(n+\frac{3}{28}\biggr) \biggl(\frac{1}{20}\biggr)^{3n}
= \frac{5\sqrt{5}}{28} \times  \frac{1}{\pi},
\end{equation}
\begin{equation}
\label{pi2}
\sum_{n=0}^\infty \binom{4n}{2n}\binom{2n}{n}^2\biggl(n+\frac{3}{40}\biggr) \biggl(\frac{1}{28}\biggr)^{4n}
=  \frac{49\,\sqrt{3}}{360} \times \frac{1}{\pi}
\end{equation}
and
\begin{equation}
\label{pi3}
\sum_{n=0}^\infty \binom{2n}{n}^3\biggl(n+\frac{1}{4}\biggr) \biggl(\frac{-1}{64}\biggr)^{n}
= \frac{1}{2\pi}.
\end{equation}
\end{Theorem}

In order to prove Theorem~\ref{aycock1}, it will be convenient to make use of the following simple lemma.

\begin{Lemma}
\label{translate}
Let $x$, $y$ and $r$ be analytic functions of a complex variable~$p$, and suppose that
$x(0)=y(0)=0$ and $r(0)=1$. Suppose that a transformation formula of the form
\begin{equation}
\label{h1}
\sum_{n=0}^\infty a(n)\, x^n = r\, \sum_{n=0}^\infty b(n)\, y^n
\end{equation}
holds in a neighborhood of $p=0$.
Let $\lambda$ be an arbitrary complex number. Then
\begin{equation}
\label{h3}
\sum_{n=0}^\infty a(n)\, (n+\lambda)\, x^n
= \sum_{n=0}^\infty b(n)\,\biggl(n\,\frac{xr}{y}\,\frac{\d y}{\d x} + \biggl(x\,\frac{\d r}{\d x}+\lambda r\biggr)\biggr)y^n.
\end{equation}
\end{Lemma}

\begin{proof}
Applying the differential operator $x\,\frac{\d}{\d x}$ to \eqref{h1} gives
\begin{equation}
\label{h2}
\sum_{n=0}^\infty a(n)\, n\, x^n = \frac{xr}{y}\,\frac{\d y}{\d x}\,\sum_{n=0}^\infty b(n)\,n\,y^n
+x\,\frac{\d r}{\d x}\,\sum_{n=0}^\infty b(n)\,y^n.
\end{equation}
Taking a linear combination of \eqref{h1} and \eqref{h2} gives the required result.
\end{proof}

We are now ready for:

\begin{proof}[Proof of Theorem \textup{\ref{aycock1}}]
By \eqref{12} and \eqref{23} in Theorem~\ref{t1}, the functions
\begin{align*}
x &= \frac {p^2 ( 1-p )^6 ( 1-4p ) ^{
6} ( 1-2p )^2 ( 1+2p )^6}
{ (1 -2p + 4p^2)^3
(1 -6p +240p^2-920p^3+960p^4-96p^5 +64p^6)^3}\\
y &={{\frac {p^3 ( 1-p )  ( 1-4p )^2
 ( 1-2p )^3 ( 1+2p ) }
 { (1 -4p +32p^3-32p^4 )^4}}
}
\intertext{and}
r &=  \frac
{ (1 -2p + 4p^2) ^{1/2}
(1 -6p +240p^2-920p^3+960p^4-96p^5 +64p^6) ^{1/2}}
{(1 -4p +32p^3-32p^4 )}
\end{align*}
satisfy the hypotheses of Lemma~\ref{translate},
where the coefficients are given by
$$
a(n) = \binom{6n}{3n}\binom{3n}{n}\binom{2n}{n}
\quad\text{and}\quad
b(n) = \binom{4n}{2n}{\binom{2n}n}^2.
$$
Setting
$$
p=\frac14(1+3\sqrt{2}-3\sqrt{3})
$$
gives
$$
x=\biggl(\frac{1}{20}\biggr)^3,\qquad y=\biggl(\frac{1}{28}\biggr)^4 \quad\text{and}\quad r=\frac{9}{28}\sqrt{10}.
$$
The derivatives can be calculated by the chain rule and we find that
\begin{align*}
\frac{\d y}{\d x}\biggr|_{p=\frac14\left(1+3\sqrt{2}-3\sqrt{3}\right)}
&=\frac{5^5 }{7^6} \times \frac{\sqrt{6}}{3}
\\ \intertext{and}
\frac{\d r}{\d x}\bigg|_{p=\frac14\left(1+3\sqrt{2}-3\sqrt{3}\right)}
&=\frac{2^2 \times 3^2 \times 5^3}{7^3} \times \sqrt{5}\,(20\sqrt{3}-21\sqrt{2}).
\end{align*}
Substituting these values in \eqref{h3} and taking $\lambda=3/28$ gives
\begin{multline*}
\sum_{n=0}^\infty \binom{6n}{3n}\binom{3n}{n}\binom{2n}{n}\biggl(n+\frac{3}{28}\biggr) \biggl(\frac{1}{20}\biggr)^{3n}
\\
= \frac{150}{343} \times \sqrt{15}\,
\sum_{n=0}^\infty \binom{4n}{2n}\binom{2n}{n}^2\biggl(n+\frac{3}{40}\biggr) \biggl(\frac{1}{28}\biggr)^{4n}.
\end{multline*}
This shows that the series evaluations \eqref{pi1} and \eqref{pi2} are equivalent.

To show that \eqref{pi2} is equivalent to \eqref{pi3}, use \eqref{23} and \eqref{4m1} as motivation to define
\begin{align*}
x &=\frac {p^3 ( 1-p )  ( 1-4p )^2 ( 1-2p )^3 ( 1+2p ) } { (1 -4p +32p^3-32p^4 )^4}, 
\\
y &=\frac {-p( 1-2p )(1+2p)^3(1-p)^3 }{( 1-4p )^6}
\intertext{and}
r &=  \frac{(1 -4p +32p^3-32p^4 )}{(1 -4p )^3}
\end{align*}
and take
$$
a(n) = \binom{4n}{2n}{\binom{2n}n}^2
\quad\text{and}\quad
b(n) = {\binom{2n}n}^3.
$$
Then compute the required derivatives, let $p$ have the same value as above, and put $\lambda=3/40$
in \eqref{h3}. We omit the details as they are similar to the above.
\end{proof}

In order to explain why the particular transformation formulas \eqref{12}, \eqref{23} and \eqref{4m1}
were used in the proof of Theorem~\ref{aycock1}, we use the classification of series, such as \eqref{pi1}--\eqref{pi3},
by modular forms. In \cite[Theorem~2.1]{cc}, it is shown that series such as \eqref{pi1}--\eqref{pi3}
can be classified according to three parameters:
\begin{enumerate}
\item
the level $\ell$;
\item
the degree $N$; and
\item
the nome $q$.
\end{enumerate}
Tables of series for $1/\pi$ for various parameter values are given in~\cite[Tables 3--6]{cc} and \cite{aldawoud}.
The relevant parameters corresponding to the series \eqref{pi1}, \eqref{pi2} and \eqref{pi3} are given by
$$
(\ell,N,q) = (1,2, e^{-2\pi\sqrt{2}}),\quad (2,9,e^{-3\pi\sqrt{2}})\quad \text{and}\quad
(4,2,-e^{-\pi\sqrt{2}}),
$$
respectively. If the $q$-parameters are denoted by $q_1$, $q_2$ and $q_3$, respectively,
then
$$
q_1=q_0^2,\quad q_2=q_0^3\quad\text{and}\quad q_3=-q_0
$$
where $q_0=e^{-\pi\sqrt{2}}$. The relevant issue is that $q_1$, $q_2$ and $q_3$ are all integral powers
of a common value $q_0$, and for $q_3$ there is also a sign change.
The values of $q_1$ and $q_2$ suggest
that quadratic and cubic transformation formulas be used, respectively, while the value of $q_3$ suggests a change of
sign is involved. The corresponding hypergeometric functions which have these properties,
for the relevant levels, are given by
\eqref{12}, \eqref{23} and \eqref{4m1}.

The entries in Tables 3--6 of \cite{cc} may be further analyzed
by their $q$-values to obtain similar relations. This leads to
the following equivalence classes of Ramanujan-type series for $1/\pi$ in Theorems~\ref{th93}--\ref{th97}.

\begin{Theorem}
\label{th93}
Ramanujan's series \textup{(30)} and \textup{(32)} in~\cite{ramanujan_pi}, namely,
\begin{equation}
\label{irrational}
\sum_{n=0}^\infty {\binom{2n} n}^3 \biggl(n+\frac{31}{270+48\sqrt{5}}\biggr) \frac{(\sqrt{5}-1)^{8n}}{2^{20n}}
= \frac{16}{15+21\sqrt{5}} \times \frac{1}{\pi}
\end{equation}
and
\begin{equation}
\label{notirrational}
\sum_{n=0}^\infty \binom{3n}n{\binom{2n}n}^2 \biggl(n+\frac{4}{33}\biggr) \frac{1}{15^{3n}}
= \frac{5\,\sqrt{3}}{22} \times  \frac{1}{\pi}
\end{equation}
are equivalent in the sense that one may be deduced from the other by using transformation formulas in
Theorem~\textup{\ref{t1}}.
\end{Theorem}

\begin{proof}
The clue is to observe that the series \eqref{irrational} and \eqref{notirrational} correspond to the
data
$$
(\ell,N,q)=(4,15,e^{-\pi\sqrt{15}}) 
\quad\text{and}\quad
(\ell,N,q)=(3,5,e^{-2\pi\sqrt{5/3}}),
$$
respectively, in the classifications in~\cite[Table 3.9]{aldawoud} and~\cite[Table 5]{cc}. The values of $q$ are related by
$$
e^{-\pi\sqrt{15}} = q_0^3
\quad\text{and}\quad
e^{-2\pi\sqrt{5/3}}=q_0^2,
\qquad\text{where}\quad
q_0=e^{-\pi\sqrt{5/3}}.
$$
Therefore, we seek a cubic transformation formula from the level~$4$ theory and a quadratic transformation
formula form the cubic theory. The relevant functions occur in~\eqref{43} and~\eqref{32}, respectively.
The proof may be completed by applying the result of Lemma~\ref{translate} and copying the
procedure in the proof of Theorem~\ref{aycock1}.
We omit the details, as they are similar, except to say that the value of $p$ in this case is given by
\begin{equation*}
p=\frac12\bigl(8-4\sqrt{3}+3\sqrt{5}-2\sqrt{15}\bigr).
\qedhere
\end{equation*}
\end{proof}

The series~\eqref{irrational} is notable for being the only one of Ramanujan's 17 examples
to contain an irrational value for the power series variable; the other 16 series all involve rational numbers.

In the remaining examples, we will be more brief.

\begin{Theorem}
The Ramanujan-type formulas for $1/\pi$ given by the following data in Tables~\textup{3--6} of \cite{cc} are equivalent:
$$
(\ell,N,q) =
(1,3, e^{-2\pi\sqrt{3}}),\;
(1,27,-e^{-3\pi\sqrt{3}}),\;
(3,9,-e^{-\pi\sqrt{3}}),\;
(4,3,e^{-\pi\sqrt{3}})
$$
and $(3,4,e^{-4\pi/\sqrt{3}})$.
\end{Theorem}

\begin{proof}
For the first four sets of parameter values, use
\eqref{12}, \eqref{1m3}, \eqref{3m} and \eqref{41}, and let $p$ be the smallest positive root of
$$
\biggl(2p+\frac{1}{2p}\biggr)^3-120\biggl(2p+\frac{1}{2p}\biggr)^2+480\biggl(2p+\frac{1}{2p}\biggr)-496=0
$$
so that $p\approx 0.00431456$.

The series corresponding to the last two sets of parameter values
$(3,4,e^{-4\pi/\sqrt{3}})$ and $(4,3,e^{-\pi\sqrt{3}})$ can be
shown to be equivalent using~\eqref{34} and~\eqref{43} and using the value
$p=1-\sqrt{3}/2$.
\end{proof}

\begin{Theorem}
The Ramanujan-type formulas for $1/\pi$ given by the following data in Tables~\textup{3--6} of \cite{cc} are equivalent:
$$
(\ell,N,q) =
(1,4, e^{-4\pi}),\;
(2,2,e^{-2\pi}),\;
(2,9,-e^{-3\pi}),\;
(4,4,-e^{-2\pi}).
$$
\end{Theorem}

\begin{proof}
Use \eqref{14}, \eqref{22}, \eqref{2m3} and \eqref{4m2}, and take
\begin{equation*}
p=\frac14\bigl(7+3\sqrt{3}-{\textstyle\sqrt{72+42\sqrt{3}}}\bigr) \approx 0.0412759.
\qedhere
\end{equation*}
\end{proof}

\begin{Theorem}
The Ramanujan-type formulas for $1/\pi$ given by the following data in Tables~\textup{3--6} of \cite{cc} are equivalent:
$$
(\ell,N,q) =
(1,7, e^{-2\pi\sqrt{7}}),\;
(1,7,-e^{-\pi\sqrt{7}}),\;
(2,7,-e^{-\pi\sqrt{7}}),\;
(4,7,e^{-\pi\sqrt{7}}).
$$
\end{Theorem}

\begin{proof}
Use \eqref{12}, \eqref{1m1}, \eqref{2m1} and \eqref{41}, and let $p$ be the smallest positive root of
\begin{multline*}
\biggl(2p+\frac{1}{2p}\biggr)^4 - 2044\biggl(2p+\frac{1}{2p}\biggr)^3
+15360\biggl(2p+\frac{1}{2p}\biggr)^2
\\
-38416\biggl(2p+\frac{1}{2p}\biggr)+31984=0
\end{multline*}
so that $p \approx 0.000245523$.
\end{proof}

\begin{Theorem}
\label{th97}
The Ramanujan-type formulas for $1/\pi$ given by the following data in Tables~\textup{3--6} of \cite{cc} are equivalent:
$$
(\ell,N,q) =
(2,3, e^{-2\pi\sqrt{3/2}}),\;
(3,2,e^{-2\pi\sqrt{2/3}}).
$$
\end{Theorem}

\begin{proof}
Use \eqref{23} and \eqref{32} and take
\begin{equation*}
p=\frac14(1+\sqrt{3}-\sqrt{6}).
\qedhere
\end{equation*}
\end{proof}

Finally, we notice that the identities in Theorems \ref{t1}, \ref{t2} and \ref{t3} can be used in designing AGM-type
algorithms \cite{agm} for the effective computation of~$\pi$ and other mathematical constants. The details of such applications can be found
in \cite{CGSZ,guillera}.

\medskip
\noindent
\textbf{Acknowledgments.}
A part of the work was done during the second author's stay at the Max-Planck-Institut f\"ur Mathematik, Bonn,
in May--June 2016. He thanks the staff of the institute for the excellent conditions he experienced when conducting this research.

\end{document}